\documentclass[12pt,letterpaper,reqno]{amsart}
\usepackage{fullpage}
\usepackage{amsmath,amsthm,amsfonts,amssymb,amscd}
\usepackage{empheq}
\usepackage{thmtools}
\usepackage{hyperref}
\usepackage{makecell}

\allowdisplaybreaks

\let\subsectiontemp\subsection
\renewcommand{\subsection}[1]{ 
    \subsectiontemp{#1} \hfill\vspace{0.5\linespacing} 
} 

\hypersetup{
  colorlinks=true,
  linkcolor=blue,
  citecolor=blue,
  linkbordercolor={0 0 1}
}

\newtheorem{theorem}{Theorem}[section]
\newtheorem{corollary}[theorem]{Corollary}
\newtheorem{lemma}[theorem]{Lemma}
\newtheorem{formula}[theorem]{Formula}

\newtheorem{conjecture}[theorem]{Conjecture}

\theoremstyle{definition}
\newtheorem{mytable}[theorem]{Table}

\mathtoolsset{showonlyrefs=true}
\numberwithin{equation}{section}
\numberwithin{figure}{section}
\numberwithin{table}{section}

\newcommand{\lrabs}[1]{\left\lvert #1 \right\lvert}
\newcommand{\lrp}[1]{\left(#1\right)}
\newcommand{\lrb}[1]{\left[#1\right]}
\newcommand{\lrfloor}[1]{\left\lfloor #1 \right\rfloor}

\DeclareMathOperator{\NN}{\mathbb{N}}

\DeclareMathOperator{\Tr}{Tr}

\newcommand{\nw}{{\textnormal{new}}}

\begin{document}

\title{Dimension sequences of modular forms}


\author[E. Ross]{Erick Ross}
\address[E. Ross]{School of Mathematical and Statistical Sciences, Clemson University, Clemson, SC}
\email{erickr@clemson.edu}

\begin{abstract}
    For $N \geq 1$, let $S_{2}^{\text{new}}(N)$ denote the newspace of cuspidal modular forms of weight $2$ and level $N$. In 2004, Greg Martin conjectured that as a sequence in $N$, $\dim S_2^{\text{new}}(N)$ takes on all possible natural numbers. In this paper, we investigate several generalizations and variations of this type of problem. In each case, we provide a complete characterization of when such a property holds.
\end{abstract}

\maketitle

\section{Introduction} \label{sec:intro}

For $N,k\geq1$, we let $S_{2k}(N)$ denote the space of cuspidal modular forms of weight $2k$ and congruence subgroup $\Gamma_0(N)$, and we let $S_{2k}^\nw(N)$ denote its new subspace.
These spaces are of great importance in Number Theory, and we would like to answer certain questions about how large they can be.
In particular, we would like to determine which of the natural numbers can be realized as dimensions of these spaces. (Here, $0$ is a natural number.)

In 2000, Csirik-Wetherell-Zieve proved that the sequence $\{\dim S_2(N)\}_{N\ge1}$ does not take on all natural numbers, interpreting these values as the genera of modular curves \cite[Section 4]{csirik}. 
In 2004, Martin considered the possible dimensions of $S_2^\nw(N)$, conjectured that the sequence $\{\dim S_2^\nw(N)\}_{N\ge1}$ takes on all possible natural numbers \cite[Conjecture 27]{martin}. 
As a corollary of certain explicit bounds on $\dim S_2^\nw(N)$, we recently disproved this conjecture by showing that $67846$ can not realized as a value of $\{\dim S_2^\nw(N)\}_{N\ge1}$ \cite[Section 7]{ross}.

In this paper, we investigate the question of when precisely results of this nature can occur, considering several generalizations and variations on the above problems. First, we consider general weight $2k$ and show that, in fact, the generalized version of Martin's conjecture can never hold.
\begin{theorem} \label{thm:N-indexed}
    For each fixed $k \ge 1$, the sequences $\{ \dim S_{2k}(N) \}_{N\ge1}$ and $\{\dim S_{2k}^\nw(N) \}_{N\ge1}$ both omit some natural number.
\end{theorem}
For each $k \ge 1$, we also determine the first omitted dimension for these sequences.

The above result for the full space $S_{2k}(N)$ is perhaps not too surprising since its dimension has a growth rate oscillating between $N$ and $N \log \log N$. However, it is much more surprising for the newspace $S_{2k}^\nw(N)$, because the growth rate in $N$ of $\dim S_{2k}^\nw(N)$ oscillates between $N / \log \log N$ and $N$.

We also consider the corresponding question for fixed level and varying weight. We give a complete classification of when these weight-indexed sequences take on every natural number.

\begin{theorem} \label{thm:k-indexed}
    For fixed $N \ge 1$, the sequence $\{ \dim S_{2k}(N) \}_{k\ge1}$ takes on every natural number 
    iff $N \in \{1,2,3,4\}$.
    Similarly, the sequence $\{ \dim S_{2k}^\nw(N) \}_{k\ge1}$ 
    takes on every natural number iff $N \in \{1,2,3,4,8,12,16,18\}$.
\end{theorem}

Finally, we consider the corresponding question on the sign pattern spaces $S_{2k}^\sigma(N)$ and $S_{2k}^{\nw,\sigma}(N)$. 
Recall that for squarefree $N$, certain operators are defined on $S_{2k}(N)$; the Atkin-Lehner involutions $W_p$ for $p \mid N$. Since the $W_p$ are involutions, they can only possibly have eigenvalues of $\pm 1$. Additionally, the $\{W_p\}_{p \mid N}$ turn out to be a commuting family of self-adjoint operators. Hence, they simultaneously diagonalize $S_{2k}(N)$, and so $S_{2k}(N)$ can be written as a direct sum of the simultaneous eigenspaces of the $W_p$. Define the sign patterns $\sigma$ for $N$ to be the nonzero multiplicative functions on the divisors of $N$ such that $\sigma(p) \in \{\pm 1\}$ for each $p \mid N$.
Then
$$
S_{2k}(N) = \bigoplus_\sigma S_{2k}^\sigma(N), \quad \text{where} \ \ S_{2k}^\sigma(N) = \{f \in S_{2k}(N) \ \colon \ W_p f = \sigma(p) f \ \ \text{for } p \mid N \}.
$$
In other words, the sign pattern spaces $S_{2k}^\sigma(N)$ are defined by specifying the sign of the $W_p$ eigenvalues for each $p \mid N$.
One can also define the sign pattern newspaces $S_{2k}^{\nw,\sigma}(N)$ in the same way \cite[Section 3]{kimball}. 



We will study the weight-indexed dimension sequences $\{ \dim S_{2k}^\sigma(N) \}_{k\ge1}$ for these sign pattern spaces. (Note that level-indexed sequences would not make sense here since the sign pattern depends on the level.) 
Of course there are many different decompositions of $S_{2k}(N)$ and $S_{2k}^\nw(N)$ we could have considered. We chose to study the sign pattern decomposition in particular because this is, in some sense, the finest possible way to decompose $S_{2k}^\nw(N)$ in a natural way (i.e. that respects the Hecke operators and Galois conjugation) \cite[Section 5]{chow-ghitza}. This gives the dimension sequences for these spaces the ``best chance" to take on every natural number. 

For $N$ and $\sigma$ fixed, we classify precisely when the dimension sequences for the sign pattern spaces take on all natural numbers. For shorthand, we will write a sign pattern $\sigma$ for $N = p_1 \cdots p_t$ (with $p_1 < \ldots < p_t$) as the sequence of signs $\sigma(p_1), \ldots, \sigma(p_t)$. For example, the sign pattern $\sigma$ for $N=30$ mapping $2 \mapsto -1$; $3 \mapsto +1$; $5 \mapsto +1$ would be written just as $\sigma = -++$.
\begin{theorem} \label{thm:k-indexed-sgnpatt}
    For fixed squarefree $N \ge 1$ with sign pattern $\sigma$, the sequence $\{ \dim S_{2k}^\sigma(N) \}_{k\ge1}$ 
    takes on every natural number iff $N \in \{1,2,3,5,6,7,10,11,14,15\}$ and $(N,\sigma) \ne (10,+-)$. 
    Similarly, the sequence $\{ \dim S_{2k}^{\nw,\sigma}(N) \}_{k\ge1}$ takes on every natural number iff 
$N \in \{
1,\allowbreak 2,\allowbreak 3,\allowbreak 5,\allowbreak 6,\allowbreak 7,\allowbreak 10,\allowbreak 14,\allowbreak 15,\allowbreak 21,\allowbreak 26,\allowbreak 30,\allowbreak 42,\allowbreak 66,\allowbreak 70,\allowbreak 78,\allowbreak 210
\}$ or 
$(N,\sigma)$ is in the list given in Table \ref{table:sporadic-cases}. 
\end{theorem}

Finally, we point out an interesting phenomenon that we observed regarding a closely related question.
One could ask about when a dimension sequence takes on every natural number exactly once. 
It turns out that for the full sign pattern spaces, whether or not
a dimension sequence satisfies this property 
depends only on its multiset density.

Given a multiset $A$ of natural numbers, the multiset density of $A$ is defined to be
$$
\rho(A) := \lim_{x \to \infty} \frac{ \#\{a \in A \colon a \le x\} }{ x },
$$
assuming this limit exists.
We note that this definition only differs from the typical notion of asymptotic density in that we are considering $A$ here to be a multiset instead of a set (i.e., we count the elements of $A$ with multiplicity). For our purposes, this will be more useful to consider than 
the typical asymptotic density since we are interested in
sequences which may have repeated values.
This also means that for sequences, the multiset density can be thought of describing a more general notion of the linear growth rate of a sequence. In the simplest case, a sequence $\{a_k\}_{k\ge1}$ of the form $a_k = \alpha k + o(k)$ will have multiset density  $1/\alpha$.

Then we have the following result, which follows from Theorem \ref{thm:k-indexed-sgnpatt}.
\begin{corollary} \label{cor:k-indexed-sgnpatt}
For fixed squarefree $N \ge 1$ with sign pattern $\sigma$, the sequence $\{ \dim S_{2k}^\sigma(N) \}_{k\ge1}$ takes on every natural number exactly once iff it has multiset density $\rho = 1$.
\end{corollary}

Because of the classification given by Theorem \ref{thm:k-indexed-sgnpatt}, this result is not all that difficult to prove.
However, it is quite surprising. 
Nothing close to this result holds for the other spaces we investigated; this structure emerges only for 
the full sign pattern dimension sequences $\{ \dim S_{2k}^\sigma(N) \}_{k\ge1}$.

This pattern is also somewhat surprising given the manner in which the 
dimension sequences in question satisfy the desired property. 
For $(N,\sigma)=(14,-+)$, the dimension sequence takes on every natural 
number exactly once in the most obvious way:\\
$\{ \dim S_{2k}^{-+}(14) \}_{k\ge1} = \{0,1,2,3,4,5,6,...\}$.
However, for other values of $(N,\sigma)$, the dimension sequences take on each natural number exactly once in non-trivial ways, such as
$\{ \dim S_{2k}^+(11) \}_{k\ge1}  = \{0,2,1,4,3,6,5,8...\}$ and  $\{ \dim S_{2k}^{+-}(15) \}_{k\ge1}  = \{1,0,3,2,5,4,7,6,...\}$.

\section{Dimension Formulas}
In this section, we state several dimension formulas for the relevant spaces of modular forms. In each case, we will write the formulas in their most explicit forms (at the loss of some elegance) in order to make for easier bounding arguments later. 

First, we give a dimension formula for the full space $S_{2k}(N)$. 
\begin{formula}[{\cite[Proposition 12]{martin}}] \label{form:fullspace}
    Let $N,k \ge 1$. Then
    \begin{align}
        \dim S_{2k}(N) = \frac{2k-1}{12} \psi(N) - \frac12 \nu_\infty(N) + c_2(k) \nu_2(N) + c_3(k) \nu_3(N) + \delta_{k,1},
    \end{align}
    where
    \begin{itemize}
        \item $c_2(k) = \lrfloor{\frac{2k}{4}} - \frac{2k-1}{4}$;
        \item $c_3(k) = \lrfloor{\frac{2k}{3}} - \frac{2k-1}{3}$;
        \item $\delta_{k,1}$ denotes the Kronecker delta;
        \item $\psi$ is the multiplicative function $\psi(N) = N \prod_{p \mid N} \lrp{1+ \frac1p}$;
        \item $\nu_\infty$ is the multiplicative function satisfying:
        \begin{itemize}
            \item $\nu_\infty(p^r) = 2 p^{(r-1)/2}$ for $r$ odd,
            \item $\nu_\infty(p^r) = (p+1)p^{r/2-1}$ for $r$ even;
        \end{itemize}
        \item $\nu_2$ is the multiplicative function satisfying:
        \begin{itemize}
            \item $\nu_2(2) = 1$ and $\nu_2(2^r) = 0$ for $r \ge 2$,
            \item $\nu_2(p^r) = 2$ when $p \equiv 1 \mod 4$,
            \item $\nu_2(p^r) = 0$ when $p \equiv 3 \mod 4$;
        \end{itemize}
        \item $\nu_3$ is the multiplicative function satisfying:
        \begin{itemize}
            \item $\nu_3(3) = 1$ and $\nu_3(3^r) = 0$ for $r \ge 2$,
            \item $\nu_3(p^r) = 2$ when $p \equiv 1 \mod 3$,
            \item $\nu_3(p^r) = 0$ when $p \equiv 2 \mod 3$.
        \end{itemize}
    \end{itemize}
\end{formula}

We also give an explicit dimension formula for the newspace $S_{2k}^\nw(N)$. 
Let $\beta$ be the multiplicative function satisfying $\beta(p) = 2$, $\beta(p^2) = -1$, and $\beta(p^r)=0$ for $r \ge 3$. Then the dimension of $S_{2k}^\nw(N)$ can be obtained as a convolution of $\beta$ with $\dim S_{2k}(N)$ \cite[Corollary 13.3.7]{cohen-stromberg}:
\begin{align}
    \dim S_{2k}^\nw(N) = \sum_{M \mid N} \beta(N/M) \dim S_{2k}(M).
\end{align}
Computing this convolution using Formula \ref{form:fullspace} yields the following explicit dimension formula for the newspace.
\begin{formula}[{\cite[Theorem 1]{martin}}] \label{form:newspace}
    Let $N,k \ge 1$. Then
    \begin{align}
        \dim S_{2k}^\nw(N) = \frac{2k-1}{12} \psi'(N) - \frac12 \nu_\infty'(N) + c_2(k) \nu'_2(N) + c_3 \nu'_3(N) + \delta_{k,1} \mu(N),
    \end{align}
    where
    \begin{itemize}
        \item $c_2(k)$, $c_3(k)$, and $ \delta_{k,1}$ are as defined in Formula \ref{form:fullspace};
        \item $\mu(N)$ is the Mobius function; 
        \item $\psi'$ is the multiplicative function satisfying:
        \begin{itemize}
            \item $\psi'(p) = p-1$, $\psi'(p^2) = p^2-p-1$, 
            and $\psi'(p^r) = \lrp{p^2-1}\lrp{p-1} p^{r-3}$ for $r \ge 3$;
        \end{itemize}
        \item $\nu'_\infty$ is the multiplicative function satisfying:
        \begin{itemize}
            \item $\nu'_\infty(p^r) = 0$ for $r$ odd,
            \item $\nu'_\infty(p^2) = p-2$ and $\nu'_\infty(p^r) = (p-1)^2 p^{r/2-2} $ for $r \ge 4$ even;
        \end{itemize}
        \item $\nu'_2$ is the multiplicative function satisfying:
        \begin{itemize}
            \item $\nu'_2(2) = -1$, $\nu_2(4) = -1$, $\nu'_2(8)=1$, and $\nu'_2(2^r) = 0$ for $r \ge 4$,
            \item $\nu'_2(p)=0$, $\nu'_2(p^2) = -1$, and $\nu'_2(p^r)=0$ for $r \ge 3$ when $p \equiv 1 \mod 4$,
            \item $\nu'_2(p)=-2$, $\nu'_2(p^2) = 1$, and $\nu'_2(p^r)=0$ for $r \ge 3$ when $p \equiv 3 \mod 4$;
        \end{itemize}
        \item $\nu'_3$ is the multiplicative function satisfying:
        \begin{itemize}
            \item $\nu'_3(3) = -1$, $\nu_3(9) = -1$, $\nu'_3(27)=1$, and $\nu'_3(3^r) = 0$ for $r \ge 4$,
            \item $\nu'_3(p)=0$, $\nu'_3(p^2) = -1$, and $\nu'_3(p^r)=0$ for $r \ge 3$ when $p \equiv 1 \mod 3$,
            \item $\nu'_3(p)=-2$, $\nu'_3(p^2) = 1$, and $\nu'_3(p^r)=0$ for $r \ge 3$ when $p \equiv 2 \mod 3$.
        \end{itemize}
    \end{itemize}
\end{formula}

Next, we give a formula to compute the dimension of the sign pattern spaces. 
For squarefree $N$, consider the trace of the Atkin-Lehner operators $\Tr W_p$ over $S_{2k}(N)$. Since the $W_p$ are simultaneously diagonalizable with eigenvalues $\pm 1$; $\Tr W_p$ gives the number of basis elements with positive $W_p$ eigenvalue, minus the number of basis elements with negative $W_p$ eigenvalue. Hence the space determined by requiring the $W_p$ eigenvalue to be $\pm1$ will have dimension $\frac12 \lrp{\Tr W_1 \pm \Tr W_p}$. ($W_1$ here is the identity operator, so $\Tr W_1 = \dim S_{2k}(N)$.) By an inclusion/exclusion argument, this idea can be extended to yield the following dimension formula for the sign pattern spaces. In the following, for $d \mid N$ not prime, the $W_d$ operators are defined multiplicatively; $W_d := \prod_{p \mid d} W_p$.
\begin{formula}[{\cite[Proposition 3.2]{kimball}}] \label{form:dim-sgnpatt-convolution}
    Let $k \ge 1$, $N \ge 1$ be squarefree, $\sigma$ be a sign pattern for $N$, and $\omega(N)$ denote the number of distinct prime divisors of $N$. Then
    \begin{align}
        \dim S_{2k}^\sigma(N) = 2^{- \omega(N)}\sum_{d \mid N} \sigma(d) \Tr W_d
    \end{align}
    The same formula also holds for the newspace $S_{2k}^{\nw,\sigma}(N)$.
\end{formula}

This formula reduces the problem of computing $\dim S_{2k}^\sigma(N)$ and $\dim S_{2k}^{\nw,\sigma}(N)$ to computing the trace of the Atkin-Lehner operators $W_d$. Explicit formulas for these traces are known \cite{yamauchi}, \cite{kimball}.
However, to avoid introducing a lot of unnecessary notation, we forgo repeating the precise trace formulas here. We will only use the following property, which follows from \cite[Equation 1.6, Proposition 1.2]{kimball}. 
\begin{lemma} \label{lem:tr-W-d-periodic}
    Fix $N \ge 1$ squarefree, and $1 \ne d \mid N$. Then for $k \ge 2$, $\Tr_{S_{2k}(N)} W_d$ and $\Tr_{S_{2k}^\nw(N)} W_d$ are $12$-periodic in $k$.
\end{lemma}
We note here that the general trace formula has an aberration at $k=1$, which is why the above periodicity result is only stated for $k \ge 2$.

\section{Level indexed dimension sequences}

In this section, we prove Theorem \ref{thm:N-indexed}, showing that each of the sequences $\{ \dim S_{2k}(N) \}_{N\ge1}$ and $\{\dim S_{2k}^\nw(N) \}_{N\ge1}$ omit some natural number.

To accomplish this, we will need explicit bounds on certain multiplicative functions which appear in the dimension formulas from the last section. We use the method from \cite[Lemma 2.4]{ross} to obtain these explicit bounds.
\begin{lemma} \label{lemma:mult-bound}
    Let $\omega(N)$ denote the number of distinct prime divisors of $N$, and let 
    $ \displaystyle
        \pi(N) := \prod_{p\mid N} 
            \frac{p^2}{p^2-p-1} 
    $. 
    Then $
        2^{\omega(N)} \leq 4.862 \, N^{1/4} 
        $
        and 
        $
        \pi(N) \le 9.930 \, N^{1/32}
        $.
\end{lemma}
\begin{proof}
    First, note that $2 \le p^{1/4}$ for primes $p \ge 17$. So let $c_p = 2 / p^{1/4}$ for primes $2 \le p \le 13$, and $c_p = 1$ for $p \ge 17$. Then
    \begin{align}
        2^{\omega(N)} = \prod_{p \mid N} 2 \le \prod_{p \mid N} c_p p^{1/4} \le c_2 c_3 \cdots c_{13} \prod_{p \mid N} p^{1/4} \le 4.862\, N^{1/4}.
    \end{align}
    Similarly, note that $\frac{p^2}{p^2 - p - 1} \le p^{1/32}$ for primes $p \ge 17$. So let $c'_p = \frac{p^2}{p^2 - p - 1} / p^{1/32}$ for primes $2 \le p \le 13$, and $c'_p = 1$ for $p \ge 17$.
    Then
    \begin{align}
        \pi(N) = \prod_{p \mid N} \frac{p^2}{p^2 - p - 1} \le \prod_{p \mid N} c'_p p^{1/32} \le c'_2 c'_3 \cdots c'_{13} \prod_{p \mid N} p^{1/32} \le 9.930\, N^{1/32},
    \end{align}
    as desired.
\end{proof}

Now, we are ready to prove Theorem \ref{thm:N-indexed}.
{
\renewcommand{\thetheorem}{\ref{thm:N-indexed}}
\begin{theorem}
    For each fixed $k \ge 1$, the sequences $\{ \dim S_{2k}(N) \}_{N\ge1}$ and $\{\dim S_{2k}^\nw(N) \}_{N\ge1}$ both omit some natural number.
\end{theorem}
\addtocounter{theorem}{-1}
}
\begin{proof}
    First, we show the result for $S_{2k}(N)$. Just by excluding the $k$ for which $0 \notin \{ \dim S_{2k}(N) \}_{N\ge1}$, we can reduce to a finite list of $k$ to consider. Recall that $\dim S_{2k}(N) \ge \dim S_{2k}(1)$, and $\dim S_{2k}(1) = 0$ only for $k \in \{1,2,3,4,5,7\}$. We then only need to consider these six remaining cases.

    Observe that in Formula \ref{form:fullspace}, we have $\lrabs{c_2(k)} \le 1/4$ and $\lrabs{c_3(k)} \le 1/3$. Additionally, we have the multiplicative bounds $\nu_\infty(N) \le 2^{\omega(N)} \sqrt{N}$, $\nu_2(N) \le 2^{\omega(N)}$, $\nu_3(N) \le 2^{\omega(N)}$.
    This means that by Formula \ref{form:fullspace},
    \begin{align}
        \dim S_{2k}(N) 
        &\ge \frac{2k-1}{12} \psi(N) - \frac12 \nu_\infty(N) - \frac12 \nu_2(N) - \frac13 \nu_3(N) \\
        &\ge \frac{2k-1}{12} \psi(N) - \frac12 \sqrt{N}2^{\omega(N)} - \frac56 2^{\omega(N)} \\ 
    \end{align}
    Then using the bounds $\psi(N) \ge N$ and $2^{\omega(N)} \le 4.862 \, N^{1/4}$ from Lemma \ref{lemma:mult-bound}, we obtain that
    \begin{align}
        \dim S_{2k}(N) \ge B_k(N) := \frac{2k-1}{12} N - \frac12 4.862\, N^{3/4} - \frac56 4.862\,  N^{1/4}.
    \end{align}
    It is straightforward to verify that (independently of $k$), $B_k(N)$ is increasing for $N \ge 240000$. We can then use $B_k(N)$ to verify by computer that a given dimension sequence omits some natural number. For $N_0 \ge 240000$ and $B \ge 0$; if $B_k(N_0) > B$, and $\dim S_{2k}(N) \ne B$ for each $1 \le N \le N_0$, then we know that the dimension sequence $\{ \dim S_{2k}(N) \}_{N\ge1}$ omits $B$.

    Via this computation, we verify that for $k=1,2,3,4,5,7$, the dimension sequences $\{ \dim S_{2k}(N) \}_{N\ge1}$ omit the values $150,23,2,4,4,1$, respectively \cite{ross-code}. For every other value of $k \ge 1$, $\{ \dim S_{2k}(N) \}_{N\ge1}$ omits $0$.

    Next, we show the result for $S_{2k}^\nw(N)$.
    In \cite[Tables 6.2]{ross}, we computed the complete list of pairs $(N,2k)$ for which $S_{2k}^\nw(N) = 0$. This reduces our problem to only a finite number of values for $k$ that we need to check. From this table, one can see that $\{S_{2k}^\nw(N)\}_{N\ge1}$ contains $0$ only for $k \in \{1,2,3,4,5,6,7\}$. 

    Recall that we defined $ \displaystyle
        \pi(N) := \prod_{p\mid N} 
            \frac{p^2}{p^2-p-1} 
    $ in Lemma \ref{lemma:mult-bound}.  Observe that in Formula \ref{form:newspace}, we have the multiplicative bounds $\psi'(N) \ge N / \pi(N)$, $\nu'_\infty(N) \le \sqrt{N}$, $\nu'_2(N) \le 2^{\omega(N)}$, and $\nu'_3(N) \le 2^{\omega(N)}$. Then using the upper bounds $2^{\omega(N)} \le 4.862 \, N^{1/4}$ and $\pi(N) \le 9.930\,N^{1/32}$, we obtain that
    \begin{align}
        \dim S_{2k}^\nw(N) \ge B'_k(N) := \frac{2k-1}{12} \frac{1}{9.930} N^{31/32} - \frac12 N^{1/2} - \frac56 4.862\, N^{1/4} - 1.
    \end{align}
    It is straightforward to verify that (independently of $k$) $B'_k(N)$ is increasing for $N \ge 4000$.

    We then can use $B_k'(N)$ to verify that for $k=1,2,3,4,5,6,7$,  the dimension sequence $\{ \dim S_{2k}^\nw(N) \}_{N\ge1}$ omits the values $67846,101,31,16,19,7,4$, respectively \cite{ross-code}. For every other value of $k \ge 1$, $\{ \dim S_{2k}(N) \}_{N\ge1}$ omits $0$.
\end{proof}

\section{Weight indexed dimension sequences}
In this section, we classify precisely when the sequences 
$\{ \dim S_{2k}(N) \}_{k\ge1}$ and $\{\dim S_{2k}^\nw(N) \}_{k\ge1}$ take on every natural number.

We will show these results by taking advantage of a very specific structure that the weight-indexed dimension sequences have. 
In particular, this structure reduces our infinite problem into a finite number of things to check. 
One hiccup is that the natural $\dim S_{2k}(N)$ formula for $k \ge 2$ breaks at $k=1$. But it turns out to not be difficult to account for this aberration.
\begin{lemma} \label{lemma:k-periodic}
    Suppose that a sequence $\{\alpha_k\}_{k\ge1}$ of natural numbers is of the form $\displaystyle \alpha_k = \frac{a}{12b} k + c(k)$ for $k \ge 2$, where $a,b \in \NN^+$ and $c(k)$ is a $12$-periodic function. Then $\{\alpha_k\}_{k\ge1}$ takes on every natural number iff $\{\alpha_k\}_{1\le k \le 24b+1}$ takes on every natural number in the range $[0,2a)$.
\end{lemma}
\begin{proof}
    First, suppose that $\{\alpha_k\}_{1\le k \le 24b+1}$ takes on every natural number in the range $[0,2a)$. Then take either $n_0 = 0$ or $n_0 = a$, chosen such that $\alpha_1 \notin [n_0,n_0+a)$. Now, for each natural number $n < n_0$, the sequence $\{\alpha_k\}_{k\ge1}$ takes on $n_0$ by assumption. And for each $n \ge n_0$, write $n = at+n'$ for $n' \in [n_0,n_0+a)$ and $t \ge 0$. Then since $n' \in [n_0,n_0+a)$, there exists some $k \ge 2$ such that $\alpha_k = n'$. Then
    \begin{align}
        \alpha_{k+12tb} &= \frac{a}{12b}(k+12tb) + c(k+12tb) 
        = \frac{a}{12b}k+c(k) +ta 
        = \alpha_k + ta = n,
    \end{align}
    and so $\{\alpha_k\}_{k\ge1}$ takes on every natural number.

    Next, suppose that $\{\alpha_k\}_{k\ge1}$ takes on every natural number. Then for each $n \in [0,2a)$, there exists a $k$ such that $\alpha_k = n$. We claim that, in fact, $1 \le k \le 24b+1$. Suppose for contradiction that $k \ge 24b+2$. Then $k-24b \ge 2$, so
    \begin{align}
        \alpha_{k-24b} = \frac{a}{12b}(k-24b) + c(k-24b)  
        = \frac{a}{12b}k+c(k) - 2a 
        = \alpha_k - 2a < 2a - 2a = 0,
    \end{align}
    which is a contradiction since $\{\alpha_k\}_{k\ge1}$ was supposed to be a sequence of natural numbers.
\end{proof}

We are now ready to prove Theorem \ref{thm:k-indexed}.
{
\renewcommand{\thetheorem}{\ref{thm:k-indexed}}
\begin{theorem} 
    For fixed $N \ge 1$, the sequence $\{ \dim S_{2k}(N) \}_{k\ge1}$ takes on every natural number 
    iff $N \in \{1,2,3,4\}$.
    Similarly, the sequence $\{ \dim S_{2k}^\nw(N) \}_{k\ge1}$ 
    takes on every natural number iff $N \in \{1,2,3,4,8,12,16,18\}$.
\end{theorem}
\addtocounter{theorem}{-1}
}
\begin{proof}
    First, we show the result for $S_{2k}(N)$. In \cite[Table 2.6]{ross}, we computed the complete list of pairs $(N,k)$ for which $\dim S_{2k}(N) = 0$. From this table, one can see that $\{ \dim S_{2k}(N) \}_{k\ge1}$ contains $0$ only for $N \in \{1, 2, 3, 4, 5, 6, 7, 8, 9, 10, 12, 13, 16, 18, 25\}$.

    Then for each of these remaining fixed values of $N$, we are interested in the sequence $\{ \dim S_{2k}(N) \}_{k\ge1}$. Observe that by Formula \ref{form:fullspace}, this sequence takes the form of Lemma \ref{lemma:k-periodic} with $a=2\psi(N)$, $b=1$.
    Thus we only need to check the finite set of values from Lemma \ref{lemma:k-periodic} to determine whether or not $\{ \dim S_{2k}(N) \}_{k\ge1}$ takes on every natural number. Checking these finite sets by computer shows that  $\{ \dim S_{2k}(N) \}_{k\ge1}$ takes on every natural number precisely for $N \in \{1,2,3,4\}$ \cite{ross-code}.

    We show the result for $S_{2k}^\nw(N)$ via an identical argument. \cite[Table 6.2]{ross} gives the complete list of pairs $(N,k)$ for which $\dim S_{2k}^\nw(N) = 0$. From this table, $\{ \dim S_{2k}^\nw(N) \}_{k\ge1}$ contains $0$ only for $N \in \{1, 2, 3, 4, 5, 6, 7, 8, 9, 10, 12, 13, 16, 18, 22, 25, 28, 60\}$.

    Then by Formula \ref{form:newspace}, $\{ \dim S_{2k}^\nw(N) \}_{k\ge1}$ takes the form of Lemma \ref{lemma:k-periodic} with $a=2\psi'(N)$, $b=1$.
    Then checking the finite sets from Lemma \ref{lemma:k-periodic} by computer shows that $\{ \dim S_{2k}^\nw(N) \}_{k\ge1}$ takes on every natural number precisely for $N \in \{1,2,3,4,8,12,16,18\}$ \cite{ross-code}.
\end{proof}

\section{Weight indexed dimension sequences with sign patterns}

In this section, we consider fixed squarefree $N$ with sign pattern $\sigma$, and classify precisely when $\{ \dim S_{2k}^\sigma(N) \}_{k\ge1}$ and $\{ \dim S_{2k}^{\nw,\sigma}(N) \}_{k\ge1}$ take on all natural numbers. 

Recall that given a multiset of natural numbers $A$, its multiset density is defined as
$$
\rho(A) := \lim_{x \to \infty} \frac{ \#\{a \in A \colon a \le x\} }{ x }.
$$
Clearly, a sequence of natural numbers $\{\alpha_k\}_{k\ge1}$ would need to have $\rho(\{\alpha_k\}_{k\ge1}) \ge 1$ in order to take on every natural number. 
So we first compute the multiset densities for the dimension sequences in question.
\begin{lemma} \label{lem:multiset-density}
    For fixed squarefree $N \ge 1$ with sign pattern $\sigma$, the dimension sequences $\{ \dim S_{2k}^\sigma(N) \}_{k\ge1}$ and $\{ \dim S_{2k}^{\nw,\sigma}(N) \}_{k\ge1}$ have multiset densities
    \begin{align}
        \rho(\{ \dim S_{2k}^\sigma(N) \}_{k\ge1}) 
        &= \frac{6 \cdot 2^{\omega(N)}}{\psi(N)}, \\
        \rho(\{ \dim S_{2k}^{\nw,\sigma}(N) \}_{k\ge1} )
        &= \frac{6 \cdot 2^{\omega(N)}}{\psi'(N)}.
    \end{align}
\end{lemma}
\begin{proof}
    Using big-$O$ notation with respect to $k$, observe that from Formula \ref{form:fullspace} and Lemma \ref{lem:tr-W-d-periodic}, we have
    \begin{align}
        \dim S_{2k}(N) = \frac{\psi(N)}{6} k + O(1) \qquad \text{and} \qquad \Tr_{S_{2k}(N)} W_d = O(1) \ \ \text{for } d \ne 1.
    \end{align}
    This means that by Formula \ref{form:dim-sgnpatt-convolution},
    \begin{align}
        \dim S_{2k}^\sigma(N) &= 2^{- \omega(N)}\sum_{d \mid N} \sigma(d) \Tr_{S_{2k}(N)} W_d \\
        &= 2^{- \omega(N)} \lrb{ \dim S_{2k}(N)  +   \sum_{1 \ne d \mid N} \sigma(d) \Tr_{S_{2k}(N)} W_d } \\ 
        &= 2^{-\omega(N)} \lrb{ \frac{\psi(N)}{6} k + O(1) } \\
        &= \frac{\psi(N)}{6 \cdot 2^{\omega(N)}} k + O(1), \label{eqn:fullspace-sgnpatt-dim-special-form}
    \end{align}
    and so $\rho(\{ \dim S_{2k}^\sigma(N) \}_{k\ge1}) 
        = \frac{6 \cdot 2^{\omega(N)}}{\psi(N)}$.
    The multiset density for the newspace dimension sequences follows by an identical argument.  
\end{proof}

This calculation then allows us to classify precisely when the sign pattern dimension sequences take on every natural number.
{
\renewcommand{\thetheorem}{\ref{thm:k-indexed-sgnpatt}}
\begin{theorem} 
    For fixed squarefree $N \ge 1$ with sign pattern $\sigma$, the sequence $\{ \dim S_{2k}^\sigma(N) \}_{k\ge1}$ 
    takes on every natural number iff $N \in \{1,2,3,5,6,7,10,11,14,15\}$ and $(N,\sigma) \ne (10,+-)$. 
    Similarly, the sequence $\{ \dim S_{2k}^{\nw,\sigma}(N) \}_{k\ge1}$ takes on every natural number iff 
$N \in \{
1,\allowbreak 2,\allowbreak 3,\allowbreak 5,\allowbreak 6,\allowbreak 7,\allowbreak 10,\allowbreak 14,\allowbreak 15,\allowbreak 21,\allowbreak 26,\allowbreak 30,\allowbreak 42,\allowbreak 66,\allowbreak 70,\allowbreak 78,\allowbreak 210
\}$ or 
$(N,\sigma)$ is in the list given in Table \ref{table:sporadic-cases}.
\end{theorem}
\addtocounter{theorem}{-1}
}
\begin{proof}
    First, we consider the full space dimension sequences $\{ \dim S_{2k}^\sigma(N) \}_{k\ge1}$. By Lemma \ref{lem:multiset-density}, these sequences have multiset density $\frac{6 \cdot 2^{\omega(N)}}{\psi(N)}$. And observe that this expression tends to $0$ fairly rapidly in $N$. (Here, $\psi(N) \ge N$ and $2^{\omega(N)} \le 4.862\, N^{1/4}$ by Lemma \ref{lemma:mult-bound}.)  It is straightforward to verify that $\frac{6 \cdot 2^{\omega(N)}}{\psi(N)} \ge 1$ only for $N \in \{1, 2, 3, 5, 6, 7, 10, 11, 14, 15\}$ \cite{ross-code}.
    We then only need to check the sequences $\{ \dim S_{2k}^\sigma(N) \}_{k\ge1}$  for this finite list of $N$. 
    
    To check these sequences, observe that each of the $O(1)$ terms in the proof of Lemma \ref{lem:multiset-density} is $12$-periodic in $k$. This means that by \eqref{eqn:fullspace-sgnpatt-dim-special-form}, $\{ \dim S_{2k}^\sigma(N) \}_{k\ge1}$ is in the form of Lemma \ref{lemma:k-periodic} with $a=2 \psi(N)$, $b = 2^{\omega(N)}$. So to determine whether or not $\{ \dim S_{2k}^\sigma(N) \}_{k\ge1}$ takes on every natural number, we only need to check the finite set of values from Lemma \ref{lemma:k-periodic}. Checking these by computer verifies that $\{ \dim S_{2k}^\sigma(N) \}_{k\ge1}$ takes on every natural number precisely for $N \in \{1, 2, 3, 5, 6, 7, 10, 11, 14, 15\}$ (with the corresponding sign patterns), except for $(N,\sigma)=(10,+-)$ \cite{ross-code}. 

    Next, we consider the newspace dimension sequences $\{ \dim S_{2k}^{\nw,\sigma}(N) \}_{k\ge1}$. Similar to before, it is straightforward to verify that $\frac{6 \cdot 2^{\omega(N)}}{\psi'(N)} \ge 1$ only for $N \in \{1,\allowbreak 2,\allowbreak 3,\allowbreak 5,\allowbreak 6,\allowbreak 7,\allowbreak 10,\allowbreak 11,\allowbreak 13,\allowbreak 14,\allowbreak 15,\allowbreak 21,\allowbreak 22,\allowbreak 26,\allowbreak 30,\allowbreak 33,\allowbreak 34,\allowbreak 35,\allowbreak 38,\allowbreak 39,\allowbreak 42,\allowbreak 46,\allowbreak 66,\allowbreak 70,\allowbreak 78,\allowbreak 102,\allowbreak 105,\allowbreak 110,\allowbreak 114,\allowbreak 130,\allowbreak 138,\allowbreak 210,\allowbreak 330,\allowbreak 390\}$ \cite{ross-code}. Then to check these remaining values of $N$, we observe that $\{ \dim S_{2k}^{\nw,\sigma}(N) \}_{k\ge1}$ is in the form of Lemma \ref{lemma:k-periodic} with $a=2 \psi'(N)$, $b = 2^{\omega(N)}$. Then checking the conditions of Lemma \ref{lemma:k-periodic} by computer verifies the desired result \cite{ross-code}.
\end{proof}

\begin{mytable} \label{table:sporadic-cases}

\begin{small}
\begin{equation}
\setlength{\arraycolsep}{1mm}
\begin{array}{|c|c|c|c|c|c|} 
\hline
\multicolumn{6}{|c|}{\makecell{
    \text{The sporadic $(N,\sigma)$ for which $\{ \dim S_{2k}^{\nw,\sigma}(N) \}_{k\ge1}$ takes on every natural number. }  
  }} \\
\hline 
(11,+) & (22,++) & (22,-+) & (22,+-) & (33,-+) & (34,++) \\
\hline 
(34,+-) & (38,++) & (38,-+) & (102,+++) & (102,--+) & (102,+--) \\
\hline 
(102,---) & (110,+++) & (110,--+) & (110,-+-) & (110,+--) & (114,+++) \\
\hline 
(114,+-+) & (114,--+) & (114,++-) & (114,-+-) & (114,+--) & (330,+-++) \\
\hline 
(330,--++) & (330,+--+) & (330,---+) & (330,+-+-) & (330,--+-) & (330,+---) \\
\hline
\end{array}
\end{equation}
\end{small}

\end{mytable}

With the classification of Theorem \ref{thm:k-indexed-sgnpatt}, we can then easily show an equivalent condition for the full space sequences taking on every natural number exactly once.

{
\renewcommand{\thecorollary}{\ref{cor:k-indexed-sgnpatt}}
\begin{corollary}
For fixed squarefree $N \ge 1$ with sign pattern $\sigma$, the sequence $\{ \dim S_{2k}^\sigma(N) \}_{k\ge1}$ takes on every natural number exactly once iff it has multiset density $\rho = 1$.
\end{corollary}
\addtocounter{theorem}{-1}
}
\begin{proof}
    Clearly, the condition $\rho = 1$ is necessary. To show that it is sufficient, we only need to check the finitely many cases of Theorem \ref{thm:k-indexed-sgnpatt} where $\rho = 1$. By Lemma \ref{lem:multiset-density}, $\rho(\{ \dim S_{2k}^\sigma(N) \}_{k\ge1}) = \frac{6\cdot 2^{\omega(N)}}{\psi(N)}$, and it is straightforward to verify that $\frac{6\cdot 2^{\omega(N)}}{\psi(N)} = 1$ only for $N \in \{11,14,15\}$. From the proof of Theorem \ref{thm:k-indexed-sgnpatt}, each of these sequences is of the form $\{ \dim S_{2k}^\sigma(N) \}_{k\ge1} = k + c(k)$, where $c(k)$ is $12$-periodic for $k\ge2$. So we can determine the general formula for each $\{ \dim S_{2k}^\sigma(N) \}_{k\ge1}$ by just computing the terms for $1\le k \le 13$. This computation \cite{ross-code} yields the following identities:
    \begin{align*}
        \{ \dim S_{2k}^{+}(11) \}_{k\ge1}
        &= \{0, 2, 1, 4, 3, 6, 5, 8, 7, 10, 9, 12, 11,\ldots\} = \{k-1+(-1)^k+\delta_{k,1}\}_{k\ge1}, \\
        \{ \dim S_{2k}^{-}(11) \}_{k\ge1}
        &= \{1, 0, 3, 2, 5, 4, 7, 6, 9, 8, 11, 10, 13,\ldots\} = \{k-1-(-1)^k\}_{k\ge1}, \\
        \{ \dim S_{2k}^{++}(14) \}_{k\ge1} 
        &= \{0, 2, 1, 4, 3, 6, 5, 8, 7, 10, 9, 12, 11,\ldots\} = \{k-1+(-1)^k+\delta_{k,1}\}_{k\ge1}, \\
        \{ \dim S_{2k}^{-+}(14) \}_{k\ge1} 
        &= \{0, 1, 2, 3, 4, 5, 6, 7, 8, 9, 10, 11, 12,\ldots\} = \{k-1\}_{k\ge1}, \\
        \{ \dim S_{2k}^{+-}(14) \}_{k\ge1} 
        &= \{1, 0, 3, 2, 5, 4, 7, 6, 9, 8, 11, 10, 13,\ldots\} = \{k-1-(-1)^k\}_{k\ge1}, \\
        \{ \dim S_{2k}^{--}(14) \}_{k\ge1} 
        &= \{0, 1, 2, 3, 4, 5, 6, 7, 8, 9, 10, 11, 12,\ldots\} = \{k-1\}_{k\ge1}, \\
        \{ \dim S_{2k}^{++}(15) \}_{k\ge1} 
        &= \{0, 2, 1, 4, 3, 6, 5, 8, 7, 10, 9, 12, 11,\ldots\} = \{k-1+(-1)^k+\delta_{k,1}\}_{k\ge1}, \\
        \{ \dim S_{2k}^{-+}(15) \}_{k\ge1} 
        &= \{0, 1, 2, 3, 4, 5, 6, 7, 8, 9, 10, 11, 12,\ldots\} = \{k-1\}_{k\ge1}, \\
        \{ \dim S_{2k}^{+-}(15) \}_{k\ge1} 
        &= \{1, 0, 3, 2, 5, 4, 7, 6, 9, 8, 11, 10, 13,\ldots\} = \{k-1-(-1)^k\}_{k\ge1}, \\
        \{ \dim S_{2k}^{--}(15) \}_{k\ge1} 
        &= \{0, 1, 2, 3, 4, 5, 6, 7, 8, 9, 10, 11, 12,\ldots\} = \{k-1\}_{k\ge1}.
    \end{align*}
    These each take on every natural number exactly once, verifying the desired result.
\end{proof}

\section{Discussion}

First, we note that the phenomenon observed in Corollary \ref{cor:k-indexed-sgnpatt} is surprising in that it only appears for the full sign pattern spaces $\dim S_{2k}^\sigma(N)$. Nothing like this result holds for any of the other families of dimension sequences we investigated. In every other family, there are plenty of examples of dimension sequences with the correct density, but which do not take on every natural number exactly once. 

As noted in the introduction, this phenomenon is also surprising in the manner that the dimension sequences take on every natural number exactly once. Observe that several of the patterns listed in the proof of Corollary \ref{cor:k-indexed-sgnpatt} take on every natural number exactly once in non-trivial ways, such as
$\{0,2,1,4,3,6,5,8...\}$ and  $\{1,0,3,2,5,4,7,6,...\}$. We would also like to point out that this structure is only possible because of the aberration at $k=1$ in the general dimension formula for $S_{2k}^\sigma(N)$. Usually, the aberration at $k=1$ in the dimension formulas for various modular forms spaces is an inconvenience. For example, it makes bounding the dimension formulas a little more annoying, and it made the proof of Lemma \ref{lemma:k-periodic} somewhat awkward. However in the case of Corollary \ref{cor:k-indexed-sgnpatt}, the aberration at $k=1$ is precisely what makes the result work.

Next, we would like to point out the strategies that made all of the results in this paper possible. For each result, we considered dimensions indexed by two parameters, $N$ and $k$. This meant that for each dimension sequence, we had to make two reductions from checking infinitely many conditions to checking only finitely many conditions. For example, to classify when $\{\dim S_{2k}(N)\}_{k\ge 1}$ took on every natural number, we first had to reduce the infinitely many sequences (for all $N\ge 1$) to only a finite list of possible $N$ to check. Second, for each of these values of $N$, we had to reduce the infinite problem of checking that $\{\dim S_{2k}(N)\}_{k\ge 1}$ took on \textit{every} natural number to a only finite problem that could be checked by computer. For each family of dimension sequences addressed in this paper, we summarize the strategy used for these two reductions.
\begin{itemize}
    \item 
    The first reduction for $\{ \dim S_{2k}(N) \}_{N\ge1}$ used the well-known level $1$ dimension formula for $\dim S_{2k}(1)$ and the fact that $\dim S_{2k}(N) \ge \dim S_{2k}(1)$. 
    \item 
    The first reduction for $\{\dim S_{2k}^\nw(N) \}_{N\ge1}$ used the table \cite[Table 6.2]{ross} of all pairs $(N,k)$ for which $\dim S_{2k}^\nw(N) = 0$.
    \item 
    The second reductions for $\{ \dim S_{2k}(N) \}_{N\ge1}$ and $\{\dim S_{2k}^\nw(N) \}_{N\ge1}$ both used explicit bounds on the dimension formula terms (namely $B_k(N)$ and $B'_k(N)$), with the help of Lemma \ref{lemma:mult-bound}.
    \item
    The first reductions for $\{ \dim S_{2k}(N) \}_{k\ge1}$ and $\{ \dim S_{2k}^\nw(N) \}_{k\ge1}$ used the tables \cite[Table 2.6]{ross} and \cite[Table 6.2]{ross} of all pairs $(N,k)$ for which $\dim S_{2k}(N) = 0$ and $\dim S_{2k}^\nw(N) = 0$.
    \item 
    The second reductions for $\{ \dim S_{2k}(N) \}_{k\ge1}$ and $\{ \dim S_{2k}^\nw(N) \}_{k\ge1}$ used Lemma \ref{lemma:k-periodic} with $(a,b) = (2\psi(N),1)$ and $(a,b) = (2\psi'(N),1)$, respectively.
    \item 
    The first reductions for $\{ \dim S_{2k}^\sigma(N) \}_{k\ge1}$ and $\{ \dim S_{2k}^{\nw,\sigma}(N) \}_{k\ge1}$ used the multiset densities of these sequences computed in Lemma \ref{lem:multiset-density}.
    \item 
    The second reductions for $\{ \dim S_{2k}^\sigma(N) \}_{k\ge1}$ and $\{ \dim S_{2k}^{\nw,\sigma}(N) \}_{k\ge1}$ used Lemma \ref{lemma:k-periodic} with $(a,b) = (2\psi(N),2^{\omega(N)})$ and $(a,b) = (2\psi'(N),2^{\omega(N)})$, respectively.
\end{itemize}

Finally, we take note of future work that still needs to be done in this area. There are several statistical properties of the dimensions of modular forms spaces that could still be investigated. For example, it was shown in \cite{csirik} that the set $\{ \dim S_2(N) \}_{N\ge 1}$ has asymptotic density $0$ in the natural numbers. However, it is not known what the density of $\{ \dim S_2^\nw(N) \}_{N\ge 1}$ is. This question was first raised to us by Greg Martin.
\begin{conjecture}
    The set $\{ \dim S_2^\nw(N) \}_{N\ge 1}$ has asymptotic density $0$.
\end{conjecture}

Additionally, further study could be done more generally regarding the statistical and asymptotic behavior of the traces of Hecke operators (in this paper, we only considered the the special case of $\Tr T_1$, which is the dimension). Some work has already been on this type of question (e.g. \cite{rouse}, \cite{chiriac}, \cite{chiriac-jorza}, \cite{nonrep}, \cite{ross-xue}, \cite{nonvan-new}, \cite{rth-coeff}), but there is still much further investigation to be done. Especially, there seems to be not much work done as of yet regarding the sign pattern spaces $S_{2k}^\sigma(N)$ and  $S_{2k}^{\nw,\sigma}(N)$.

\begin{flushright} S.D.G.\end{flushright}

\section*{Data availability statement}
All code for this paper is publicly available on GitHub \cite{ross-code}.

\bibliographystyle{plain}
\bibliography{bibliography.bib}

\end{document}